\documentclass[12pt]{article}
\usepackage{graphicx,epsfig}
\usepackage{color}
\usepackage{amsmath,amssymb,mathrsfs,amsthm}
\usepackage{amsfonts}
\usepackage{multicol}
 \usepackage{graphicx}
\usepackage{color}
\usepackage{pdfcolmk}
\usepackage[nice]{nicefrac}
\usepackage{amsthm}
\usepackage{latexsym}
\usepackage{dsfont}
\usepackage{setspace}
\usepackage[a4paper, tmargin=1.5in, bmargin=1in, lmargin=1in, rmargin=1in, headheight=13.6pt]{geometry}
\usepackage[sort&compress,numbers]{natbib}
\usepackage[font=small,labelfont=bf,labelsep=space]{caption}
\makeatletter

\date{}

\def\@citex[#1]#2{\if@filesw\immediate\write\@auxout{\string\citation{#2}}\fi
  \def\@citea{}\@cite{\@for\@citeb:=#2\do
    {\@citea\def\@citea{,\linebreak[0]\hskip0pt plus .2em}%
      \@ifundefined{b@\@citeb}%
    {{\bf ?}\@warning{Citation `\@citeb' on page \thepage\space undefined}}%
      \hbox{\csname b@\@citeb\endcsname}}}{#1}}

\newtheorem{theorem}{Theorem}[section]

\newtheorem{remark}{Remark}[section]

\newtheorem{example}{Example}[section]

\newtheorem{rule-def}[theorem]{Rule}
\numberwithin{equation}{section}
\makeatletter

\begin{document}
\title{Optimal homotopy analysis method with Green's function for a class of nonlocal elliptic boundary value problems}
\author{Randhir Singh \thanks{Corresponding author. E-mail:randhir.math@gmail.com}\\
$\rm $\small {Department of Mathematics}\\
\small {Birla Institute of Technology Mesra,
 Ranchi-835215, India}} \maketitle{}
\begin{abstract}
\noindent
In this paper, we present the optimal homotopy analysis method (OHAM) with Green's function technique to acquire accurate numerical
solutions for the nonlocal elliptic problems. We first transform the nonlocal boundary value problems into an equivalent integral equation, and then use an OHAM with convergence control parameter $c_0$. To demonstrate convergence and accuracy characteristics of the OHAM method, we compare the OHAM and Adomian decomposition method (ADM) with Green's function. The numerical experiments confirm the reliability of the approach as it handles such nonlocal elliptic differential equations without imposing limiting assumptions that could change the physical structure of the solution. We also discuss the convergence and error analysis of proposed method.  In summary: $(i)$ the present approach  does not require any additional computational work for unknown constants unlike ADM and  VIM \cite{khuri2014variational} $(ii)$  guarantee of convergence $(iii)$ flexibility on choice of  initial guess of solution and $(iv)$ useful analytic tool to investigate a class of nonlocal elliptic boundary value problems.

\end{abstract}
\textbf{Keyword}: Optimal homotopy analysis method;  Nonlinear nonlocal elliptic boundary value problems; Convergence analysis; Adomian decomposition method; Integral equations.

\section{Introduction}
We first consider a class of linear nonlocal elliptic boundary value problems \cite{cannon2008numerical,khuri2014variational}:
\begin{align}\label{sec1:e1}
\left\{
  \begin{array}{ll}
\displaystyle  -\alpha\bigg(\int\limits_{0}^{1} y(s)ds \bigg) y''(x)=h(x),~~~~~x\in (0,1)\\
y(0)=a,~~~~y(1)=b,~~~~~a,b \in [0,\infty).
\end{array} \right.
\end{align}
This problem \eqref{sec1:e1} may be called a class of  linear nonlocal  boundary value problem since the coefficient of the derivative of the unknown solution $y$ depends upon the integral of $y$ itself, which in turn depends on the whole domain $(0,1)$ rather than on a single point.

We also study a class of nonlinear nonlocal elliptic nonlinear boundary value problems \cite{cannon2011numerical,khuri2014variational,themistoclakis2015numerical}:
\begin{align}\label{sec1:e2}
\left\{
  \begin{array}{ll}
\displaystyle  -\alpha\bigg(\int\limits_{0}^{1} y(s)ds \bigg) y''(x)+y^{2n+1}(x)=0,~~~~~x\in (0,1)\\
y(0)=a,~~~~y(1)=b,~~~~~a,b \in [0,\infty),~~ n\in 0\cup\mathds{Z}^{+}
\end{array} \right.
\end{align}
The problem  \eqref{sec1:e2} may be called a class of  nonlinear nonlocal  boundary value problem. Such nonlocal  boundary value problems arise in modeling various physical questions such as the aero-elastic behavior of suspended flexible cables subjected to icing conditions and wind action  \cite{luongo1998non,luongo2008continuous,themistoclakis2015numerical} or the dust production and diffusion in the fusion devices \cite{themistoclakis2015numerical}. For details on such applications of nonlocal  boundary value problems  see \cite{stanczy2001nonlocal,khuri2014variational} and the references therein.

In  \cite{cannon2008numerical}, the existence and uniqueness of solution of \eqref{sec1:e1} was discussed by using a fixed point theorem and then  the numerical solutions were obtained via a finite difference scheme. In  \cite{cannon2011numerical}, Cannon and
Galiffa developed a numerical method for  \eqref{sec1:e2}, in which they established a priori estimates and the existence and uniqueness
of the solution to the nonlinear auxiliary problem via the Schauder fixed point theorem. They proved the existence and uniqueness to the problem and analyzed a discretization of  problem  \eqref{sec1:e2} and showed that a solution to the nonlinear difference problem exists and is unique and that the numerical procedure converges with error $O(h)$. In \cite{khuri2014variational}, Khuri and  Wazwaz  applied the variational iteration method to a class
of  nonlocal, elliptic boundary value problems \eqref{sec1:e1} and \eqref{sec1:e2} and established  uniform convergence
of the scheme. In \cite{themistoclakis2015numerical}, Themistoclakis and Vecchio provided the sufficient conditions for the unique solvability and a more general convergence theorem for  \eqref{sec1:e2} and suggested different iterative procedures to handle the nonlocal nonlinearity of the discrete problem.

In this paper, we present the OHAM with Green's function technique to acquire accurate numerical
solutions for the nonlocal elliptic problems given in \eqref{sec1:e1} and \eqref{sec1:e2}. We first transform the given nonlocal boundary value problems into an equivalent integral equation, and then use OHAM  to obtain accurate numerical
solutions. To demonstrate convergence and accuracy characteristics of the OHAM, a number of test examples are included. We compare the present method and ADM with Green's function which confirms the accuracy and superiority  of the OHAM. The numerical experiments confirm the reliability of the approach as it handles such nonlocal elliptic differential equations without imposing limiting assumptions that could change the physical structure of the solution. We also discuss the convergence and error analysis of proposed method.  In summary: $(i)$ the present approach  does not require any additional computational work for unknown constants unlike ADM and  VIM \cite{khuri2014variational} $(ii)$  guarantee of convergence $(iii)$ flexibility on choice of an  initial guess of solution and $(iv)$ useful analytic tool to investigate a class of nonlocal elliptic boundary value problems.

\section{Homotopy analysis method with Green's function}
We consider a general  form of  \eqref{sec1:e1} and \eqref{sec1:e2} nonlocal elliptic boundary value problems as
\begin{align}\label{sec2:e1}
\left\{
  \begin{array}{ll}
\displaystyle \alpha (p) y''(x)=f(y(x))~~~~~x\in (0,1)\\
y(0)=a,~~~~y(1)=b,~~~~\hbox{where}~~~~p=\displaystyle\int\limits_{0}^{1} y(s)ds,
\end{array} \right.
\end{align}
where $\alpha(p)$ is a continuous positive function. By setting  $f(y(x))\equiv -h(x)$ or $f(y(x)) \equiv y^{2n+1}(x)$, we obtain the original boundary value problems  \eqref{sec1:e1} or \eqref{sec1:e2}.

Following  Singh et al. \cite{singh2014new,singh2015approximate}, we transform nonlocal elliptic boundary value problems \eqref{sec2:e1} into an the equivalent integral equation as
\begin{align}\label{sec2:e2}
 y(x)=a+(b-a)x+\frac{1}{\alpha (p)} \int\limits_{0}^{1} G(x,s) f(y(s))ds,
\end{align}
where $G(x,s)$ is given by
\begin{align} \label{sec2:e3}
 G(x,\xi)=\left  \{
  \begin{array}{ll}
   x (s-1), & \hbox{$ x\leq s$}, \\
  s (x-1) , & \hbox{$s\leq x$}.
\end{array}
\right.
\end{align}
According to homotopy analysis method \cite{liao2007general,liao2009series}, we use  $q \in[0, 1]$ as an embedding parameter, the general zero-order deformation equation is constructed as
\begin{align}\label{sec2:e4}
(1-q)[\phi(x;q)-y_0(x)]=q\; c_0\;  N[\phi(x;q)],
\end{align}
where  $y_0(x)$ denotes an initial guess,  $c_0\neq0$ is convergence-controller parameter,   $\phi(x;q)$ is an unknown function and $N[\phi(x;q)]$ is given by
\begin{align}\label{sec2:e5}
N[\phi(x;q)]:=\phi(x;q)-\big[a+(b-a)x\big]-\frac{1}{\alpha(p[\phi(x;q)])} \int\limits_{0}^{1}G(x,s)  f(\phi(s;q))ds=0.
\end{align}
The zero-order deformation \eqref{sec2:e4} becomes $\phi(x;0)=y_0(x)$ at  $q = 0$ and
it becomes $N[\phi(x;1)]=0$ at $q = 1$ which is exactly the same as the original problem \eqref{sec2:e1} provided that $\phi(x;1)= y(x)$.

Expanding $\phi(x;q)$ in a Taylor series with respect to the parameter $q$, we obtain
\begin{align}\label{sec2:e7}
\phi(x;q)=y_0(x)+\sum_{k=1}^{\infty} y_k(x) q^{k},
\end{align}
where $y_k(x)$ is given by
\begin{align}\label{sec2:e8}
y_k(x)=\frac{1}{k!}\frac{\partial^k }{\partial q^k}[\phi(x;q)]\bigg|_{q=0}.
\end{align}
The series \eqref{sec2:e7}  converges for $q = 1$ if $c_0\neq0$ is chosen properly and it  reduces to
 \begin{align}\label{sec2:e9}
\phi(x;1)\equiv y(x)=\sum_{k=0}^{\infty}  y_k(x),
\end{align}
which will be one of solutions of the problem \eqref{sec2:e2}.

Defining the vector $\overrightarrow{y_k} = \{y_0 (x), y_1 (x), \ldots, y_k (x)\}$ and differentiating \eqref{sec2:e4},  $k$ times with respect to the parameter $q$, dividing it by $k!$,
setting subsequently $q = 0,$ the $k$th-order deformation equation is obtained
\begin{align}\label{sec2:e10}
y_k(x)-\chi_{k}\; y_{k-1}(x)=c_0\ R_k(\overrightarrow{y}_{k-1},x),
\end{align}
where
\begin{align}\label{sec2:e11}
\chi_{k}=\left   \{
  \begin{array}{ll}
    0, & \hbox{$k\leq1$} \\
     1, & \hbox{$k>1$},
\end{array}
\right.
\end{align}
and
\begin{align*}
R_k(\overrightarrow{y}_{k-1},x)&=\frac{1}{(k-1)!} \bigg\{\frac{\partial^{k-1}}{\partial q^{k-1}} N[\phi(x;q)]\bigg\}_{q=0}\\
                  &=\frac{1}{(k-1)!} \bigg \{ \frac{\partial^{k-1}} {\partial q^{k-1}} N\bigg( \displaystyle \sum_{m=0}^{\infty} y_m(x) q^{m}\bigg)\bigg\}_{q=0}\\
                  &=y_{k-1}(x)-(1-\chi_{k})\big[a+(b-a)x\big]-\frac{1}{\alpha(p_{k-1})}\int\limits_{0}^{1}G(x,s) \mathcal{H}_{k-1} ds.
\end{align*}
Thus we have
\begin{align}\label{sec2:e12}
R_k(\overrightarrow{y}_{k-1},x)=y_{k-1}(x)-(1-\chi_{k})\big[a+(b-a)x\big]-\frac{1}{\alpha(p_{k-1})}\int\limits_{0}^{1}G(x,s) \mathcal{H}_{k-1} ds
\end{align}
where $p_{k}$ and $\mathcal{H}_{k}$, are given by
\begin{align}\label{sec2:e13}
p_{k}=\displaystyle \int\limits_{0}^{1} \bigg(\sum_{m=0}^{k} y_m(x) q^{m}\bigg) ds,~~~\mathcal{H}_{k}=\frac{1}{k!}\frac{\partial^{k} }{\partial q^{k}} \bigg\{ f\bigg(\sum_{m=0}^{k} y_m(x) q^{m}\bigg)\bigg\}_{q=0}.
\end{align}
Using \eqref{sec2:e10}- \eqref{sec2:e13}, the $m$th-order deformation equation is simplified as
\begin{align}\label{sec2:e14}
\nonumber y_k(x)-\chi_{k} y_{k-1}(x)&=c_0\ \bigg[y_{k-1}(x)-(1-\chi_{k})\big[a+(b-a)x\big]\\
&~~~~~~~~~~~~~~~~~~-\frac{1}{\alpha(p_{k-1})} \int\limits_{0}^{1}G(x,s) \mathcal{H}_{k-1} ds\bigg].
\end{align}
Choosing  an initial guess $y_0(x)=a+(b-a)x$, the  components  $y_k$ are successively obtained. Hence, the $n$th-order approximate solution of the problem \eqref{sec2:e2} is given by
\begin{align}\label{sec2:e17}
\phi_n(x,c_0)=\sum_{k=0}^{n} y_k(x,c_0).
\end{align}
The optimal value of the parameter $c_0$ can be obtained by minimizing the squared residual of governing equation
\begin{align}\label{sec2:e18}
E_{n}(c_0)=  \int_{0}^{1}  \big(N[\phi_n(x,c_0)]\big)^2  dx.
\end{align}
However, the exact squared residual error is expensive to calculate when $n$ is large. So, we
approximate $E_n$ by  the discrete averaged residual error defined by
\begin{align}\label{sec2:e19}
E_{n}(c_0)\approx \frac{1}{M}\sum_{k=1}^{M}  \big(N[\phi_n(x_k,c_0)]\big)^2,
\end{align}
where  $x_k = kh,$ $k=1,2,...M$ and the optimal value $c_0$ is obtained by solving $\frac{dE_{n}}{dc_0}=0$.
\begin{remark}
\rm{By setting $c_0=-1$, the scheme \eqref{sec2:e14} reduces to the ADM with Green's function \cite{singh2014new,singh2015approximate}.}
\end{remark}

\section{Convergence analysis}
In this section, we establish the convergence of solution defined in \eqref{sec2:e17} of  integral \eqref{sec2:e2}. Let $\mathds{X}= \big(C[0,1], \|y\|\big)$ be a Banach space with $$\|y\|=\max_{ x\in [0,1]} |y(x)|,~y\in \mathds{X}.$$

\begin{theorem}\label{sec3:eq1}
\rm{Let $0 < \delta< 1$ and the solution components $y_0(x),y_1(x),y_2(x),\ldots$ obtained by \eqref{sec2:e17} satisfy the  condition:
$\exists~~ k_0 \in \mathbb{N}~~\forall~ k\geq k_0:~ \|y_{k+1}\|\leq \delta\|y_{k}\|,$ then the series solution $\sum_{k=0}^{\infty} y_k(x)$ is convergent.}
\end{theorem}
\begin{proof}
Define the sequence $\{\phi_n\}_{n=0}^{\infty}$ as,
\begin{align}\label{sec3:eq3}
\left\{
  \begin{array}{ll}
\phi_0=y_0(x)\\
\phi_1=y_0(x)+y_1(x)\\
\phi_2=y_0(x)+y_1(x)+y_2(x)\\
\vdots\\
\phi_n=y_0(x)+y_1(x)+y_2(x)+ \cdots +y_n(x)\\
\end{array} \right.
\end{align}
and we show that is a Cauchy sequence in the Banach space  $\mathds{X}$. For this purpose, consider
\begin{align*}
\|\phi_{n+1}-\phi_{n}\|&=\|y_{n+1}\|\leq \delta \|y_{n}\|\leq  \delta^2 \|y_{n-1}\|\leq \ldots \leq \delta^{n-k_0+1} \|y_{k_{0}}\|.
\end{align*}
For every $n,m\in \mathbb{N}$, $n\geq m >k_0$, we have
\begin{align}\label{sec3:eq4}
\nonumber \|\phi_{n}-\phi_{m}\|&=\|(\phi_n-\phi_{n-1})+(\phi_{n-1}-\phi_{n-2})+\cdots+(\phi_{m+1}-\phi_{m})\|\\
\nonumber&\leq \|\phi_n-\phi_{n-1}\|+\|\phi_{n-1}-\phi_{n-2}\|+\cdots+\|\phi_{m+1}-\phi_{m}\|\\
\nonumber &\leq   (\delta^{n-k_0} + \delta^{n-k_0-1} +\cdots+\delta^{m-k_0+1})\|y_{k_{0}}\|\\
&= \frac{1-\delta^{n-m}}{1-\delta}\delta^{m-k_{0}+1}   \|y_{k_{0}}\|
\end{align}
and since $0<\delta<1$ so it follows that
\begin{eqnarray}\label{sec3:eq5}
\lim_{n,m\rightarrow \infty} \|\phi_{n}-\phi_{m}\|=0.
\end{eqnarray}
Therefore, $\{\phi_n\}_{n=0}^{\infty}$ is a Cauchy sequence in the Banach space  $\mathds{X}$.
\end{proof}

\begin{theorem}\label{sec3:eq6}
\rm{Assume that the series solution $\sum_{k=0}^{\infty} y_k(x)$  converges  to $y(x)$. If the truncated
series $\phi_m(x,c_0)=\sum_{k=0}^{m} y_k(x,c_0)$ is used as an approximation to the solution $y(x)$, then the maximum absolute truncated error is estimated as
\begin{align}
\left|y(x)-\phi_m(x,c_0)\right|\leq\frac{\delta^{m-k_{0}+1} }{1-\delta}  \|y_{k_{0}}\|.
\end{align}
}
\end{theorem}
\begin{proof}
From Theorem \ref{sec3:eq1}, following inequality \eqref{sec3:eq4}, we have
\begin{align*}
\|\phi_{n}-\phi_m(x,c_0)\|\leq\frac{1-\delta^{n-m}}{1-\delta}\delta^{m-k_{0}+1}   \|y_{k_{0}}\|,
\end{align*}
for $n\geq m$. Now, as $n\rightarrow \infty$ then   $\phi_{n}\rightarrow y$ and $\delta^{n-m}\rightarrow0$. So,
\begin{align}\label{sec3:eq7}
\|y(x)-\phi_m(x,c_0)\|\leq\frac{\delta^{m-k_{0}+1}}{1-\delta}   \|y_{k_{0}}\|.
\end{align}
Theorems \ref{sec3:eq1} and \ref{sec3:eq6} together confirm that the convergence of series solution \eqref{sec2:e17}.
\end{proof}

\begin{theorem}\label{sec3:eq11}
\rm{If  the series solution $\sum_{k=0}^{\infty} y_k(x)$ converges  to  $y(x)$ then it must be a solution of  \eqref{sec2:e2}.}
\end{theorem}
\begin{proof}
Since the series  $\sum_{k=0}^{\infty} y_k(x)$  is convergent, then
\begin{align}\label{sec3:eq12}
\lim_{n\rightarrow\infty} y_n(x)=0,~~~~ \forall~~x\in[0,1].
\end{align}
By summing up the left hand-side of   \eqref{sec2:e10}, we get
\begin{align}\label{sec3:eq13}
\sum_{k=1}^{n}[y_k(x)-\chi_{k} y_{k-1}(x)]=y_1(x)+\ldots+(y_n(x)-y_{n-1}(x))=y_n(x).
\end{align}
Letting $n\rightarrow\infty$ and using \eqref{sec3:eq12}, equation \eqref{sec3:eq13} reduces to
\begin{align}\label{sec3:eq14}
\sum_{k=1}^{\infty}(y_k(x)-\chi_{k} y_{k-1}(x))=0.
\end{align}
Using \eqref{sec3:eq14}  and right hand-side of the relation    \eqref{sec2:e10},  we obtain
\begin{align}\label{sec3:eq15}
\sum_{k=1}^{\infty} c_0\ R_k(\overrightarrow{y}_{k-1},x)=\sum_{k=1}^{\infty}(y_k(x)-\chi_{k} y_{k-1}(x))=0.
\end{align}
Since $c_0\neq 0$, then equation \eqref{sec3:eq15} reduces to
\begin{align}\label{sec3:eq16}
\sum_{k=1}^{\infty} \ R_k(\overrightarrow{y}_{k-1},x)=0.
\end{align}
 Using \eqref{sec3:eq16} and  \eqref{sec2:e10}, we have
\begin{align*}
0=\sum_{k=1}^{\infty} \ R_k(\overrightarrow{y}_{k-1},x)&=\sum_{k=1}^{\infty} \bigg[y_{k-1}(x)-(1-\chi_{k})\big[a+(b-a)x\big]-\frac{1}{\alpha(p_{k-1})}\int\limits_{0}^{1}G(x,s) \mathcal{H}_{k-1} ds\bigg]\\
 &=  \sum_{k=1}^{\infty} y_{k-1}(x)-\big[a+(b-a)x\big]-\frac{1}{ \displaystyle \sum_{k=1}^{\infty}\alpha(p_{k-1})}\int\limits_{0}^{1}G(x,s)   \sum_{k=1}^{\infty} \mathcal{H}_{k-1}  ds,
  \end{align*}
since $\sum_{k=0}^{\infty} y_k(x)$  converges to $y(x)$, then  $\sum_{k=0}^{\infty} \mathcal{H}_{k}\rightarrow f(y(x))$ and $\sum_{k=1}^{\infty}\alpha(p_{k-1})\rightarrow\alpha (p)$ \cite{cherruault1989convergence}, we obtain
\begin{align*}
y(x)=a+(b-a)x+\frac{1}{\alpha (p)} \int\limits_{0}^{1}G(x,s) f(y(s))  ds.
\end{align*}
Hence, $y(x)$ is the exact solution of  integral equation \eqref{sec2:e2}.
\end{proof}

\section{Numerical results}
In this section, four examples are
discussed and the results are compared with existing exact solutions. We define the absolute  errors as
  \begin{align*}
 E_{n}(x):= |y(x)-\phi_{n}(x)|,~~~~~~e_{n}(x):= |y(x)-\psi_{n}(x)|
\end{align*}
where $ \phi_{n}(x)$ and $\psi_{n}(x)$ are OHAM and  ADM solutions, respectively.
\begin{example}\label{sec4:ex1}
\rm{Consider the  special case of linear nonlocal elliptic boundary value problem \eqref{sec1:e1} with $\alpha(p)=p^{1/3}$ as
\begin{align}
\left\{
  \begin{array}{ll}
\displaystyle  p^{1/3} y''(x)=\frac{6}{\sqrt[3]{4}}x,~~~~~~~~x\in (0,1)\\
y(0)=0,~~~~y(1)=1,~~~~~p=\bigg(\int\limits_{0}^{1} y(s)ds\bigg).
\end{array} \right.
\end{align}
Its exact solution is $y(x)=x^3$.}
\end{example}
Applying the OHAM \eqref{sec2:e14}  to the example \eqref{sec4:ex1}, we have
\begin{align}\label{sec4:ex2}
y_k(x)-\chi_{k} y_{k-1}(x)=c_0\ \bigg[y_{k-1}(x)-(1-\chi_{k})y_0(x)-\frac{1}{\alpha(p_{k-1})} \int\limits_{0}^{1}G(x,s) \frac{6s}{\sqrt[3]{4}} ds\bigg].
\end{align}
Using \eqref{sec4:ex2} with an initial approximation $y_0=x$, we obtain $\phi_2(x,c_0)$. With the help of \eqref{sec2:e19}, the optimal value of the parameter $c_0$ is computed as $c_0=-0.505595$, and hence the OHAM solution is given by
\begin{align}
\phi_2(x)= 7.7\times 10^{-16} x+x^3.
\end{align}
A comparison among the numerical results obtained by OHAM solution $\phi_{2}(x)$, ADM solution $\psi_{2}(x)$  and the exact solution $y(x)$ is depicted in Table \ref{tab1}.

\begin{table}[htbp]
\caption{Numerical results of example   \ref{sec4:ex1}}
\vspace{-0.3cm}
\centering
\setlength{\tabcolsep}{0.08in}
\begin{tabular}{l|ccc| cc}
\hline
\cline{1-6}
$x$ & $y(x)$  & $\psi_{2}(x,-1)$ & $\phi_{2}(x,-0.505595)$  & $|y(x)-\psi_{2}(x)|$ & $|y(x)-\phi_{2}(x)|$ \\
\cline{1-6}
0.0	&	0.000	&	0.0000000000	&	0.000	&	0.000000000	&	0.00000000	\\
0.1	&	0.001	&	-0.062559671	&	0.001	&	0.063559671	&	7.67615E-17	\\
0.2	&	0.008	&	-0.115267240    &	0.008	&	0.123267240	&	1.49186E-16	\\
0.3	&	0.027	&	-0.148270607	&	0.027	&	0.175270607	&	2.11636E-16	\\
0.4	&	0.064	&	-0.151717670	&	0.064	&	0.215717670	&	2.49800E-16	\\
0.5	&	0.125	&	-0.115756328	&	0.125	&	0.240756328	&	2.77556E-16	\\
0.6	&	0.216	&	-0.030534480	&	0.216	&	0.246534480	&	3.05311E-16	\\
0.7	&	0.343	&	0.1137999760	&	0.343	&	0.229200024	&	2.77556E-16	\\
0.8	&	0.512	&	0.3270991400	&	0.512	&	0.184900860	&	2.22045E-16	\\
0.9	&	0.729	&	0.6192151140	&	0.729	&	0.109784886	&	1.11022E-16	\\
1.0	&	1.000	&	1.0000000000	&	1.000	&	2.22045E-16	&	0.000000000	\\
\hline
\end{tabular}
\label{tab1}
\end{table}

\begin{example}\label{sec4:ex3}
\rm{Consider the special case of nonlinear nonlocal elliptic boundary value problem  \eqref{sec1:e2} with $\alpha(p)=\frac{1}{p}$ as
\begin{align}
\left\{
  \begin{array}{ll}
\displaystyle  -\frac{1}{p}y''(x)+\frac{3}{4(2\sqrt{2}-2)}y^5(x)=0,~~~~~~~~x\in (0,1)\\
y(0)=1,~~~~y(1)=\frac{\sqrt{2}}{2},~~~~~p=\bigg(\int\limits_{0}^{1} y(s)ds\bigg).
\end{array} \right.
\end{align}
Its  exact solution is $y(x)=\frac{1}{\sqrt{1+x}}$.}
\end{example}
Applying the OHAM \eqref{sec2:e14} to the example \eqref{sec4:ex3}, we have
\begin{align}\label{sec4:ex4}
y_k(x)-\chi_{k} y_{k-1}(x)=c_0\ \bigg[y_{k-1}(x)-(1-\chi_{k})y_0(x)-\frac{1}{\alpha(p_{k-1})} \int\limits_{0}^{1}G(x,s) \mathcal{H}_{k-1} ds\bigg].
\end{align}
Using \eqref{sec4:ex4} with an initial guess $y_0=1+(\frac{\sqrt{2}}{2}-1)x$ and \eqref{sec2:e19} with $c_0=-0.819014$, we obtain the homotopy optimal solution as
\begin{align*}
\phi_2(x)&=1-0.4973x+0.3737x^2-0.3060 x^3+0.2235x^4-0.1258 x^5+0.05148x^6\\
&~~~-0.0153 x^7+0.0033x^8-0.00055x^9+0.0000646x^{10}+\cdots
\end{align*}
A comparison among the numerical solution obtained by OHAM solution $\phi_{2}(x)$, ADM solution $\psi_{2}(x)$  and the exact solution is depicted in Table \ref{tab2}.

\begin{table}[htbp]
\caption{Numerical results of example   \ref{sec4:ex3}}
\vspace{-0.3cm}
\centering
\setlength{\tabcolsep}{0.08in}
\begin{tabular}{l|ccc| cc}
\hline
\cline{1-6}
$x$ & $y(x)$  & $\psi_{2}(x,c_0=-1)$ & $\phi_{2}(x,-0.819014)$  & $|y(x)-\psi_{2}(x)|$ & $|y(x)-\phi_{2}(x)|$ \\
\cline{1-6}
0.0	&	1.000000000	&	1.000000000	&	1.000000000	&	0.000000000	&	0.000000000	\\
0.1	&	0.953462589	&	0.954516555	&	0.953715758	&	0.001053966	&	0.000253169	\\
0.2	&	0.912870929	&	0.914909713	&	0.913348055	&	0.002038784	&	0.000477126	\\
0.3	&	0.877058019	&	0.879819116	&	0.877702187	&	0.002761097	&	0.000644168	\\
0.4	&	0.845154255	&	0.848286083	&	0.845886767	&	0.003131828	&	0.000732512	\\
0.5	&	0.816496581	&	0.819638873	&	0.817233417	&	0.003142293	&	0.000736836	\\
0.6	&	0.790569415	&	0.793406404	&	0.791235825	&	0.002836989	&	0.000666410	\\
0.7	&	0.766964989	&	0.769254375	&	0.767504100	&	0.002289386	&	0.000539111	\\
0.8	&	0.745355992	&	0.746938889	&	0.745731089	&	0.001582896	&	0.000375097	\\
0.9	&	0.725476250	&	0.726273545	&	0.725667948	&	0.000797295	&	0.000191698	\\
1.0	&	0.707106781	&	0.707106781	&	0.707106781	&	0.000000000	&	2.22045E-16	\\
\hline
\end{tabular}
\label{tab2}
\end{table}



\begin{example}\label{sec4:ex5}
\rm{Consider the special case of nonlinear nonlocal elliptic boundary value problem \eqref{sec1:e2} with $\alpha(p)=p$  as
\begin{align}
\left\{
  \begin{array}{ll}
\displaystyle  -p \;y''(x)+\frac{3 (2\sqrt{2}-2)}{4}y^5(x)=0,~~~~~~~~x\in (0,1)\\
y(0)=1,~~~~y(1)=\frac{\sqrt{2}}{2},~~~~~p=\bigg(\int\limits_{0}^{1} y(s)ds\bigg).
\end{array} \right.
\end{align}
Its exact solution is $y(x)=\frac{1}{\sqrt{1+x}}$.}
\end{example}
Applying the OHAM \eqref{sec2:e14} with $y_0=1+(\frac{\sqrt{2}}{2}-1)x$, we have
\begin{align}\label{sec4:ex6}
y_k(x)-\chi_{k} y_{k-1}(x)=c_0\ \bigg[y_{k-1}(x)-(1-\chi_{k})y_0(x)-\frac{1}{\alpha(p_{k-1})} \int\limits_{0}^{1}G(x,s) \mathcal{H}_{k-1} ds\bigg].
\end{align}
Using \eqref{sec4:ex6} and \eqref{sec2:e19} with $c_0=-0.933697$, we obtain the homotopy optimal solution as
\begin{align*}
\phi_{2}(x)&=1-0.495211 x+0.362361 x^2-0.280911 x^3+0.195035 x^4-0.107115 x^5+0.043453 x^6\\
&~~~~~~-0.01294 x^7+0.002854 x^8-0.000464 x^9+0.000054 x^{10}-\cdots
\end{align*}
A comparison among the numerical solution obtained by OHAM solution $\phi_{2}(x)$, ADM $\psi_{2}(x)$  and the exact solution is depicted in Table \ref{tab3}.

\begin{table}[htbp]
\caption{Numerical results of example \ref{sec4:ex5}}
\vspace{-0.3cm}
\centering
\setlength{\tabcolsep}{0.08in}
\begin{tabular}{l|ccc| cc}
\hline
\cline{1-6}
$x$ & $y(x)$  & $\psi_{2}(x,-1)$ & $\phi_{2}(x,-0.933697)$  & $|y(x)-\psi_{2}(x)|$ & $|y(x)-\phi_{2}(x)|$ \\
\cline{1-6}
0.0	&	1.000000000	&	1.000000000	&	1.000000000	&	0.000000000	&	0.000000000	\\
0.1	&	0.953462589	&	0.954139200	&	0.953840089	&	0.000676611	&	0.000377500	\\
0.2	&	0.912870929	&	0.914044009	&	0.913485384	&	0.001173080	&	0.000614454	\\
0.3	&	0.877058019	&	0.878552757	&	0.877813154	&	0.001494738	&	0.000755135	\\
0.4	&	0.845154255	&	0.846797874	&	0.845969674	&	0.001643619	&	0.000815419	\\
0.5	&	0.816496581	&	0.818128147	&	0.817301384	&	0.001631566	&	0.000804803	\\
0.6	&	0.790569415	&	0.792049712	&	0.791302438	&	0.001480297	&	0.000733023	\\
0.7	&	0.766964989	&	0.768181846	&	0.767575192	&	0.001216857	&	0.000610203	\\
0.8	&	0.745355992	&	0.746224331	&	0.745800843	&	0.000868338	&	0.000444851	\\
0.9	&	0.725476250	&	0.725933776	&	0.725717917	&	0.000457526	&	0.000241667	\\
1	&	0.707106781	&	0.707106781	&	0.707106781	&	0.000000000	&	2.22045E-16	\\
\hline
\end{tabular}
\label{tab3}
\end{table}


\begin{example}\label{sec4:ex7}
\rm{Consider the special case of nonlinear nonlocal elliptic boundary value problem \eqref{sec1:e2} with $\alpha(p)=\bigg(\frac{1}{p}\bigg)^{2}$  as
\begin{align}
\left\{
  \begin{array}{ll}
\displaystyle  -\bigg(\frac{1}{p}\bigg)^{2} y''(x)+\frac{2}{(\ln2)^2}y^3(x)=0,~~~~~~~~x\in (0,1)\\
y(0)=1,~~y(1)=\frac{1}{2},~~~~~p=\bigg(\int\limits_{0}^{1} y(s)ds\bigg).
\end{array} \right.
\end{align}
Its exact solution is $y(x)=\frac{1}{1+x}$.}
\end{example}
Applying the OHAM \eqref{sec2:e14} with $y_0=1+(\frac{1}{2}-1)x$, we have
\begin{align}\label{sec4:ex8}
y_k(x)-\chi_{k} y_{k-1}(x)=c_0\ \bigg[y_{k-1}(x)-(1-\chi_{k})y_0(x)-\frac{1}{\alpha(p_{k-1})} \int\limits_{0}^{1}G(x,s) \mathcal{H}_{k-1} ds\bigg].
\end{align}
Using \eqref{sec4:ex8} and \eqref{sec2:e19} with $c_0=-0.612671$, we obtain the homotopy optimal solution as
\begin{align*}
\phi_2(x)&=1-1.00399 x+0.995127 x^2-0.90086 x^3+0.655261 x^4-0.338986 x^5+0.115228 x^6\\
&~~~~~~-0.0246916 x^7+0.00308645 x^8-0.00017147 x^9+\cdots
\end{align*}
A comparison among the numerical solution obtained by OHAM solution $\phi_{2}(x)$, ADM solution $\psi_{2}(x)$  and the exact solution is depicted in Table \ref{tab4}.

\begin{table}[htbp]
\caption{Numerical results of example  \ref{sec4:ex7}}
\vspace{-0.3cm}
\centering
\setlength{\tabcolsep}{0.08in}
\begin{tabular}{l|ccc| cc}
\hline
\cline{1-6}
$x$ & $y(x)$  & $\psi_{2}(x,-1)$ & $\phi_{2}(x,-0.612671)$  & $|y(x)-\psi_{2}(x)|$ & $|y(x)-\phi_{2}(x)|$ \\
\cline{1-6}
0.0	&	1.000000000	&	1.000000000	&	1.000000000	&	0.000000000	&	0.000000000	\\
0.1	&	0.909090909	&	0.914550054	&	0.908713383	&	0.005459145	&	0.000377526	\\
0.2	&	0.833333333	&	0.844849352	&	0.832746652	&	0.011516019	&	0.000586682	\\
0.3	&	0.769230769	&	0.785573674	&	0.768603045	&	0.016342905	&	0.000627724	\\
0.4	&	0.714285714	&	0.733317575	&	0.713705107	&	0.019031861	&	0.000580607	\\
0.5	&	0.666666667	&	0.686029523	&	0.666157571	&	0.019362857	&	0.000509096	\\
0.6	&	0.625000000	&	0.642573190	&	0.624561470	&	0.017573190	&	0.000438530	\\
0.7	&	0.588235294	&	0.602393994	&	0.587870965	&	0.014158699	&	0.000364329	\\
0.8	&	0.555555556	&	0.565272431	&	0.555285414	&	0.009716875	&	0.000270141	\\
0.9	&	0.526315789	&	0.531148042	&	0.526170109	&	0.004832252	&	0.000145681	\\
1.0	&	0.500000000	&	0.500000000 &	0.500000000 &	6.66134E-16 &	1.66533E-16	\\
\hline
\end{tabular}
\label{tab4}
\end{table}


\section{Conclusion}
We presented the OHAM with Green's function technique for obtaining numerical solutions
to a class of nonlinear, nonlocal, elliptic boundary value problems. We first transformed the given nonlocal boundary value problems into an equivalent integral equation, and then used the OHAM to obtain accurate solutions. Unlike the ADM, the  OHAM  always gives better convergent series solution as  shown in Tables. The numerical experiments confirm the reliability of the approach as it handles such nonlocal elliptic differential equations without imposing limiting assumptions that could change the physical structure of the solution. In addition, the convergence and error analysis was presented. The proposed scheme will be used further in studying identical applications.


\begin{thebibliography}{10}
\expandafter\ifx\csname url\endcsname\relax
  \def\url#1{\texttt{#1}}\fi
\expandafter\ifx\csname urlprefix\endcsname\relax\def\urlprefix{URL }\fi
\expandafter\ifx\csname href\endcsname\relax
  \def\href#1#2{#2} \def\path#1{#1}\fi

\bibitem{khuri2014variational}
S.~Khuri, A.M. Wazwaz, A variational approach for a class of nonlocal elliptic
  boundary value problems, Journal of Mathematical Chemistry 52~(5) (2014)
  1324--1337.

\bibitem{cannon2008numerical}
J.~R. Cannon, D.~J. Galiffa, et~al., A numerical method for a nonlocal elliptic
  boundary value problem, Journal of Integral Equations and Applications 20~(2)
  (2008) 243--261.

\bibitem{cannon2011numerical}
J.~R. Cannon, D.~J. Galiffa, On a numerical method for a homogeneous,
  nonlinear, nonlocal, elliptic boundary value problem, Nonlinear Analysis:
  Theory, Methods \& Applications 74~(5) (2011) 1702--1713.

\bibitem{themistoclakis2015numerical}
W.~Themistoclakis, A.~Vecchio, On the numerical solution of some nonlinear and
  nonlocal boundary value problems, Applied Mathematics and Computation 255
  (2015) 135--146.

\bibitem{luongo1998non}
A.~Luongo, G.~Piccardo, Non-linear galloping of sagged cables in 1: 2 internal
  resonance, Journal of Sound and Vibration 214~(5) (1998) 915--940.

\bibitem{luongo2008continuous}
A.~Luongo, G.~Piccardo, A continuous approach to the aeroelastic stability of
  suspended cables in 1: 2 internal resonance, Journal of Vibration and Control
  14~(1-2) (2008) 135--157.

\bibitem{stanczy2001nonlocal}
R.~Sta{\'n}czy, Nonlocal elliptic equations, Nonlinear Analysis: Theory,
  Methods \& Applications 47~(5) (2001) 3579--3584.

\bibitem{singh2014new}
R.~Singh, G.~Nelakanti, J.~Kumar, A new efficient technique for solving
  two-point boundary value problems for integro-differential equations, Journal
  of Mathematical Chemistry 52~(8) (2014) 2030--2051.

\bibitem{singh2015approximate}
R.~Singh, G.~Nelakanti, J.~Kumar, Approximate solution of two-point boundary
  value problems using Adomian decomposition method with Green's function,
  Proceedings of the National Academy of Sciences, India Section A: Physical
  Sciences 85~(1) (2015) 51--61.

\bibitem{liao2007general}
S.~Liao, Y.~Tan, A general approach to obtain series solutions of nonlinear
  differential equations, Studies in Applied Mathematics 119~(4) (2007)
  297--354.

\bibitem{liao2009series}
S.~Liao, Series solution of nonlinear eigenvalue problems by means of the
  homotopy analysis method, Nonlinear Analysis: Real World Applications 10~(4)
  (2009) 2455--2470.

\bibitem{cherruault1989convergence}
Y.~Cherruault, Convergence of Adomian's method, Kybernetes 18~(2) (1989)
  31--38.

\end{thebibliography}

\end{document}